\documentclass[11pt]{article}
\usepackage[T1]{fontenc}
\usepackage[utf8]{inputenc}
\usepackage{lmodern}
\usepackage[margin=0.95in]{geometry}
\usepackage{amsmath,amssymb,amsthm,mathtools}
\usepackage{booktabs,array,longtable}
\usepackage{enumitem}
\usepackage{microtype}
\usepackage[hidelinks]{hyperref}
\hypersetup{pdfauthor={Navid Rezazadeh},pdftitle={Covariance-Rank Rounding and a Sharp Integrality Gap for a Matroid Edge-Coverage Moment Lift},pdfsubject={Formal research manuscript for independent mathematical review},pdfkeywords={matroid, coverage, approximation, research draft}}
\usepackage{fancyhdr}
\usepackage{listings}
\setlist{topsep=4pt,itemsep=2pt,parsep=0pt}
\newtheorem{theorem}{Theorem}[section]
\newtheorem{lemma}[theorem]{Lemma}
\newtheorem{proposition}[theorem]{Proposition}
\newtheorem{corollary}[theorem]{Corollary}
\theoremstyle{definition}
\newtheorem{definition}[theorem]{Definition}
\theoremstyle{remark}
\newtheorem{remark}[theorem]{Remark}
\newcommand{\R}{\mathbb R}
\newcommand{\Q}{\mathbb Q}
\newcommand{\E}{\mathbb E}
\newcommand{\one}{\mathbf 1}
\newcommand{\OPT}{\operatorname{OPT}}
\newcommand{\tr}{\operatorname{tr}}
\newcommand{\rank}{\operatorname{rank}}
\newcommand{\conv}{\operatorname{conv}}
\newcommand{\cl}{\operatorname{cl}}
\newcommand{\diag}{\operatorname{diag}}
\newcommand{\cB}{\mathcal B}
\newcommand{\cI}{\mathcal I}
\newcommand{\cR}{\mathcal R}
\newcommand{\ip}[2]{\langle #1,#2\rangle}

\numberwithin{equation}{section}
\title{Covariance-Rank Rounding and a Sharp Integrality Gap\texorpdfstring{\\}{ }for a Matroid Edge-Coverage Moment Lift}
\author{Navid Rezazadeh\\
  {\small UC Irvine}\\
  {\small\href{mailto:nrezazad@uci.edu}{\texttt{nrezazad@uci.edu}}}}
\date{10 October 2026}
\begin{document}
\maketitle
\vspace{-1.3em}
\begin{abstract}
Selecting vertices covers their incident edges, but a matroid can prevent selecting the combinations that cover the most weight. We study whether pair-selection information improves rounding from a semidefinite relaxation of this problem. A feasible point $(x,Y)$ records individual and paired selection values. Its relaxation objective is $U$, and the matrix $Y-xx^\top$ has rank $d$. We compare the feasible marginal vectors obtained by fixing one element as selected or excluded, then round only the best vector. For $d\ge1$, the resulting basis has value at least $(3/4+1/(4d))U$. The procedure is deterministic and polynomial-time for a supplied exact rational point and an independence oracle. Weighting the covariance spectrum by the objective gives a stronger version of the bound. To show the rank factor is sharp, we construct a rank-$2m$ rational linear matroid and $m$ disjoint four-cycles with integral optimum $3m+1$, relaxation optimum $4m$, and an optimal covariance matrix of rank $m$. Thus the relaxation has limiting reciprocal integrality gap $4/3$, despite its conditional constraints. One valid pair-exclusion inequality makes it exact on this family, as an explicit semidefinite certificate shows. The result concerns rounding and the strength of this relaxation; it does not supply a low-rank optimal point or a fixed universal approximation improvement.
\end{abstract}

\section{Introduction and scope}
A selected vertex covers every edge incident to it. We seek a set whose covered edge weight is largest, subject to that set being independent in a matroid. Write $f(S)$ for its covered weight and $\OPT$ for the maximum feasible value. Because all weights are nonnegative, an optimum can be a basis, that is, a maximal independent set. For this coverage objective, Huang and Sellier~\cite{HS23} record a polynomial-time $3/4$ approximation based on a coverage linear program and matroid rounding. The general monotone submodular threshold is therefore not the relevant baseline.

Our question is about a stronger upper bound: how much of a feasible semidefinite relaxation value can an integral basis recover? The relaxation records a vector $x$ of singleton-selection values and a matrix $Y$ of pair-selection values. Its constraints require feasible marginal vectors after either fixing or excluding any one element. These local requirements do not ensure that all entries come from a single distribution over bases. The challenge is to use the available pair information without making that unjustified assumption.

The central step is to form the two marginal vectors associated with each element and compare their multilinear coverage values. This value is the expected covered weight under independent selection at the given marginals. We choose the vector with the largest value. Standard exchange rounding then produces a basis without decreasing that objective. The proof evaluates an average over these choices; the algorithm itself is deterministic. Correlations in the matrix determine how much this single-coordinate step improves on the $3/4$ baseline.

\subsection{The pointwise guarantee and its limit}
For a feasible lifted point $(x,Y)$, let $U(x,Y)$ be its relaxation objective and let $d$ be the rank of its covariance matrix $Y-xx^\top$. We prove
\begin{equation}\label{eq:intro-bound}
 f(S)\ge\left(\frac34+\frac{1}{4d}\right)U(x,Y),
 \qquad d=\rank(Y-xx^\top)\ge1.
\end{equation}
The rank $d$ describes the supplied matrix, not the matroid. The algorithm is polynomial-time for an exact rational feasible point. A stronger result uses how the objective distributes weight across the covariance directions. Neither statement asserts that a low-rank optimal point can be found efficiently.

The distinction between the relaxation and the optimization problem is essential. Let $U^\star$ be the optimal relaxation value. For every positive integer $m$, we construct an instance satisfying
\begin{equation}\label{eq:intro-gap}
 U^\star=4m,\qquad \OPT=3m+1,\qquad d=m.
\end{equation}
The ratio $(3m+1)/(4m)$ attains the rank-dependent factor and approaches $3/4$. Consequently, no rounding method can guarantee $(3/4+\delta)U^\star$ for a fixed $\delta>0$ on all instances: its requested value would exceed $\OPT$ on this family. This is a limitation of the upper bound, not an approximation-hardness theorem for the underlying problem.

For $m>1$, the construction permits pair statistics that are individually compatible with feasible bases but cannot hold simultaneously. One pair-exclusion inequality makes the relaxation exact on the family, as a semidefinite upper-bound certificate shows.

\subsection{Relation to established methods}
The relaxation is a normalized form of the positive semidefinite matrix lift of Lov\'asz and Schrijver~\cite{LS91}. It retains the pair variables rather than projecting onto $x$; the lift is not new. The final step uses established matroid exchange rounding~\cite{CCPV11,CVZ09}. The complete-bipartite construction and mutually exclusive two-element pairs also occur in a determinant-maximization obstruction~\cite[Appendix F]{MNST20}. Section~\ref{sec:monomials} compares the relaxation with the stronger matroid monomial formulations of Fischer, Fischer, and McCormick~\cite{FFM22}.

Rezazadeh and Kia~\cite{RezazadehKia23} studied distributed monotone submodular maximization under partition-matroid constraints using multilinear extensions and stochastic pipage rounding. Their related work~\cite{RezazadehKia22} examined the tradeoff between approximation performance and privacy in distributed submodular maximization. Both papers concern distributed optimization, rather than the conditional moment lift and integrality gap analyzed here.

The claims requiring a priority comparison are the pointwise covariance/effective-rank bound and its matching example for this conditional lift. We do not claim novelty for the lift, exchange rounding, or the pair-exclusion principle. The results below do not establish historical priority.

Section~\ref{sec:definitions} defines the lift and the quantities the proof uses. Section~\ref{sec:main} develops the conditioning argument. Sections~\ref{sec:gap} and~\ref{sec:cuts} establish sharpness and identify the missing inequality. The numerical qualifications in Section~\ref{sec:implementation} distinguish an exact supplied-point theorem from claims about semidefinite solver output.

\section{Definitions and the lifted problem}\label{sec:definitions}
We first specify the feasible sets and the three objective values used in the paper: the integral value, its multilinear extension, and the lifted value. The lift then supplies the covariance identities needed for conditioning.

\subsection{Matroids and coverage}
Nonnegative coverage permits optimization over bases rather than all independent sets. We define this setting before the relaxation.

A matroid is a pair $M=(V,\cI)$. The family $\cI\subseteq2^V$ satisfies the following properties. The empty set is in $\cI$. Every subset of a member of $\cI$ is in $\cI$. If $A,B\in\cI$ and $|A|<|B|$, some $b\in B\setminus A$ satisfies $A\cup\{b\}\in\cI$.

A basis is an inclusion-maximal independent set. All bases have the same size, denoted by $r$. Write $\cB$ for the family of bases. For $A\subseteq V$, the rank function is
\[
 \rho(A)=\max\{|I|:I\subseteq A,\ I\in\cI\}.
\]
Write $A+v=A\cup\{v\}$ and $A-v=A\setminus\{v\}$. 

For algorithmic access, an independence oracle returns whether a supplied set belongs to $\cI$. For the explicit constructions, a linear matroid is specified by columns over a field. Its independent sets are the linearly independent column subsets.

Let $G=(V,H)$ be a finite undirected graph without loops. Let $w_{uv}\in\Q_{\ge0}$ for each edge $uv\in H$. Each unordered edge occurs once. Write $n=|V|$ and $W=\sum_{uv\in H}w_{uv}$. Put
\begin{equation}\label{eq:coverage}
 f(S)=\sum_{uv\in H}w_{uv}\one[\{u,v\}\cap S\ne\varnothing],
 \qquad
 \OPT=\max_{S\in\cI}f(S).
\end{equation}
The notation $\one[P]$ is one when statement $P$ holds, and zero otherwise.

The function $f$ satisfies $f(\varnothing)=0$ and is monotone. It is also submodular. Indeed, the marginal gain from adding $v$ equals the weight of its incident edges that are not yet covered. This quantity cannot increase when the selected set grows. Every independent set extends to a basis. Monotonicity therefore implies that a basis attains $\OPT$.

For $z\in[0,1]^V$, define
\begin{equation}\label{eq:F}
 F(z)=\sum_{uv\in H}w_{uv}(z_u+z_v-z_uz_v).
\end{equation}
Independent selection with marginal probabilities $z$ has expected value $F(z)$. In particular, $F(\one_S)=f(S)$.

The base polytope and its scaled copies are
\[
 B(M)=\conv\{\one_B:B\in\cB\},\qquad
 aB(M)=\{az:z\in B(M)\}\quad(a\ge0).
\]
Here $\one_B$ is the indicator vector of $B$. We use $0B(M)=\{0\}$. For any vector $z$, write $z(A)=\sum_{v\in A}z_v$.

\begin{table}[t]
\centering
\small
\begin{tabular}{@{}p{0.23\linewidth}p{0.71\linewidth}@{}}
\toprule
Symbol & Definition \\
\midrule
$n$, $r$ & Ground-set size $|V|$ and matroid rank $\rho(V)$ \\
$W$ & Total edge weight $\sum_{uv\in H}w_{uv}$ \\
$Z$ & Augmented moment matrix $\left(\begin{smallmatrix}1&x^\top\\x&Y\end{smallmatrix}\right)$ \\
$C$, $d$ & $C=Y-xx^\top$ and $d=\rank(C)$ \\
$U(x,Y)$, $U^\star$ & Lifted objective value and its maximum \\
$\sigma_v$ & $\sqrt{x_v(1-x_v)}$ \\
$T$ & $\sum_{uv\in H}w_{uv}\sigma_u\sigma_v$ \\
$A$, $D$ & Positive semidefinite matrices defined in \eqref{eq:AD} \\
$\beta$ & $\tr(A^2)/\tr(A)^2$, when $T>0$ \\
$d_{\rm act}$ & Rank of $C[J,J]$ on the active index set in \eqref{eq:active} \\
\bottomrule
\end{tabular}
\caption{Notation reference for the rounding theorem. Later quantities are defined in the text when needed. Matrix rank and matroid rank are different quantities.}
\label{tab:notation}
\end{table}

\subsection{The matrix lift}
The lift enforces feasible conditional marginals without requiring true joint moments. This distinction drives both the rounding algorithm and the gap.

A real symmetric matrix $Q$ is positive semidefinite, written $Q\succeq0$, when $a^\top Qa\ge0$ for every real vector $a$ of the appropriate dimension. The trace $\tr(Q)$ is the sum of its diagonal entries. Matrix rank is the dimension of its column space.

Define $\cR(M)$ as the set of pairs $(x,Y)\in\R^V\times\R^{V\times V}$ that satisfy
\begin{equation*}\label{eq:R}\tag{R}
\begin{aligned}
 x&\in B(M), &Y&=Y^\top,\\
 Y_{vv}&=x_v &&(v\in V),\\
 Z:=\begin{pmatrix}1&x^\top\\x&Y\end{pmatrix}&\succeq0,\\
 Y_{\cdot v}&\in x_vB(M) &&(v\in V),\\
 x-Y_{\cdot v}&\in(1-x_v)B(M) &&(v\in V).
\end{aligned}
\end{equation*}
The column $Y_{\cdot v}$ includes its diagonal entry. In particular, this system uses the full base polytope in every conditional constraint. For $0<x_v<1$, dividing the last two constraints by $x_v$ and $1-x_v$ gives two points of $B(M)$. These are the conditional candidates used for rounding; the zero-scaling convention handles endpoint coordinates.

The lifted objective is
\begin{equation}\label{eq:U}
 U(x,Y)=\sum_{uv\in H}w_{uv}(x_u+x_v-Y_{uv}),
 \qquad
 U^\star=\max_{(x,Y)\in\cR(M)}U(x,Y).
\end{equation} For true selection moments, $x_u+x_v-Y_{uv}$ is the probability that at least one endpoint is selected.

The next proposition establishes the upper bound, attainment of $U^\star$, and the covariance identities needed for conditioning. Here $\one$ denotes the all-ones vector.

\begin{proposition}[Validity, compactness, and covariance identities]\label{prop:lift}
The set $\cR(M)$ is nonempty, convex, and compact. Every distribution on bases gives a feasible pair of first and second moments. Thus $U^\star\ge\OPT$. For every feasible point,
\begin{align}
 C:=Y-xx^\top&\succeq0, & C_{vv}&=x_v(1-x_v),\label{eq:covdiag}\\
 \max(0,x_u+x_v-1)&\le Y_{uv}\le\min(x_u,x_v),\label{eq:frechet}\\
 C\one&=0.\label{eq:rowsum}
\end{align}
If $x_v\in\{0,1\}$, then $C_{uv}=0$ for every $u$. If $n\ge1$, then $\rank(C)\le n-1$.
\end{proposition}
\begin{proof}
Let $B$ be a random basis and put $X=\one_B$. Set $x=\E X$ and $Y=\E[XX^\top]$. Then $x\in B(M)$, $Y_{vv}=x_v$, and
\[
 Z=\E\!\left[\begin{pmatrix}1\\X\end{pmatrix}
                 \begin{pmatrix}1\\X\end{pmatrix}^{\!\top}\right]\succeq0.
\]
When $x_v>0$, the vector $Y_{\cdot v}/x_v$ is the conditional mean of $X$ given $X_v=1$. It is a convex combination of basis indicators and therefore lies in $B(M)$. The complementary conditional mean gives the last constraint. When an event has probability zero, its unnormalized column is zero. This proves feasibility in all cases. Moreover, $U(x,Y)=\E f(B)$.

Every vector in $B(M)$ has coordinates in $[0,1]$. The two scaled-column constraints give
\[
 0\le Y_{uv}\le x_v,
 \qquad
 0\le x_u-Y_{uv}\le1-x_v.
\]
Symmetry gives \eqref{eq:frechet}. These inequalities bound all variables.

The cone $K=\{(a,z):a\ge0,\ z\in aB(M)\}$ is closed and convex. For closedness, take a convergent sequence $(a_j,a_jb_j)$ with $b_j\in B(M)$. Compactness of $B(M)$ gives a convergent subsequence of $b_j$. Its limit verifies membership, including when $a_j$ tends to zero. Convexity follows by taking weighted averages in $B(M)$. 

Each scaled-column constraint is the inverse image of $K$ under a linear map. The other constraints are closed and convex. Thus $\cR(M)$ is closed, convex, and bounded. A fixed basis supplies a feasible point, so it is nonempty. The continuous objective attains a maximum.

The block identity
\[
 Z=
 \begin{pmatrix}1&0\\x&I\end{pmatrix}
 \begin{pmatrix}1&0\\0&C\end{pmatrix}
 \begin{pmatrix}1&x^\top\\0&I\end{pmatrix}
\]
proves $C\succeq0$, because the outer factors are invertible transposes of one another. The diagonal of $C$ gives \eqref{eq:covdiag}. A $2\times2$ principal minor gives
$C_{uv}^2\le C_{uu}C_{vv}$. This proves the zero-row assertion.

Finally, the entries of column $v$ sum to $rx_v$, because $Y_{\cdot v}\in x_vB(M)$. By symmetry, $Y\one=rx$. Since $\one^\top x=r$, we obtain $C\one=0$. For $n\ge1$, the all-ones vector is nonzero, so this kernel vector gives $\rank(C)\le n-1$.
\end{proof}

To identify the established lift, use the cone $K$ and its scale coordinate. Here $\diag(Z)$ is the vector of diagonal entries.

\begin{remark}[The exact Lov\'asz--Schrijver interpretation]\label{rem:LS}
Let $e_0,e_1,\ldots,e_n$ be the standard coordinate vectors in the augmented space. The normalized matrix lift of $K=\{(a,z):a\ge0,\ z\in aB(M)\}$ requires
\[
 Z\succeq0,\quad Z_{00}=1,\quad Ze_0=\diag(Z),\quad
 Ze_v\in K,\quad Z(e_0-e_v)\in K.
\]
These are the positive semidefinite matrix-cone conditions of Lov\'asz and Schrijver~\cite[Section 1]{LS91}. Substitution of $Z$ gives \eqref{eq:R}. Its $x$-projection is $B(M)$: one inclusion is explicit, and the other follows from Proposition~\ref{prop:lift}. We retain $Y$ because the objective depends on second moments. We do not identify this system with a lift of an extended coverage linear program that has extra edge variables.
\end{remark}

\subsection{Which pair moments the relaxation permits}\label{sec:monomials}
The distinction between true moments and the lifted variables explains what the rounding theorem must overcome. 

Fischer, Fischer, and McCormick~\cite{FFM22} study polynomial objectives over
matroid independent sets by linearizing monomials. They describe a strengthened
rank-inequality system and a hierarchy of monomial formulations. In particular,
Corollary~13 of their 2017 author preprint removes the independence requirement
on the underlying variable set. All monomials on that set are retained, with
those supported on dependent sets fixed to zero. This distinction matters for
our family, where the union of all quadratic supports is not independent.

Let $H_2(M)$ denote the convex hull of the singleton and pair indicators of bases:
\[
 H_2(M)=\operatorname{conv}\{(\mathbf1_B,(\mathbf1_B)_u(\mathbf1_B)_v)_{u<v}:B\in\mathcal B\}.
\]
Projecting the full monomial convex hull, restricted to the base equation
$x(V)=r$, onto singleton and quadratic coordinates gives exactly $H_2(M)$.
In these singleton and pair coordinates, the true moment polytope is contained in our relaxation $\mathcal R(M)$:
every distribution over bases satisfies~\eqref{eq:R}. The containment can
be strict, as Theorem~\ref{thm:gap} shows.

Consequently, the exact full-monomial hull restricted to bases implies every
inequality valid for $H_2(M)$, including the pair-exclusion inequality in
Section~\ref{sec:cuts}. This follows from the projection identity above, not
from an assertion that the cited paper states that particular inequality.
No claim about a specified truncated level is needed here. The pair-exclusion
principle is not new.

The candidate contribution is a pointwise bound and matching gap for this weaker conditional PSD system. The monomial formulations inspected do not state the same covariance-rank theorem or example; this comparison does not establish priority.

\section{Rounding a feasible marginal vector}\label{sec:rounding}
Conditioning produces points in $B(M)$. The next guarantee rounds them while preserving $F$; Appendix~\ref{app:constructive-rounding} supplies the constructive proof. A $t$-term decomposition is a convex combination of $t$ basis indicators.

\begin{lemma}[Objective-preserving deterministic rounding]\label{lem:round}
Given $z\in B(M)\cap\Q^V$, a deterministic polynomial-time procedure returns a basis $B$ with $f(B)\ge F(z)$. Given a decomposition with $t$ terms, the merging stage uses at most $(t-1)r$ exchanges and $O(tr^2)$ independence queries.
\end{lemma}
\begin{proof}
Appendix~\ref{app:constructive-rounding} gives the complete decomposition, symmetric-exchange, and rational-arithmetic argument. Along a feasible exchange direction, $F$ is convex, so one endpoint preserves its current value.
\end{proof}

This standard exchange-rounding step~\cite{CCPV11,CVZ09} uses exact evaluations of $F$, not value-oracle sampling. The algorithm calls it once.

\section{The covariance-rank theorem}\label{sec:main}
The goal is a better pointwise ratio when covariance occupies fewer directions. We first state the guarantee, then show how feasible conditional marginals, two edge inequalities, and a spectral average establish it.

\begin{theorem}[Pointwise covariance-rank rounding]\label{thm:main}
Let $(x,Y)\in\cR(M)$, and put $d=\rank(Y-xx^\top)$. If $d=0$, then $x$ is a basis indicator and $f(\{v:x_v=1\})=U(x,Y)$. If $d\ge1$, some basis $S$ satisfies
\begin{equation}\label{eq:main}
 f(S)\ge\left(\frac34+\frac{1}{4d}\right)U(x,Y).
\end{equation}
For exact rational $(x,Y)$ and rational nonnegative edge weights, the basis can be found deterministically in polynomial time. The same procedure satisfies the stronger effective-rank bound in Theorem~\ref{thm:effective}.
\end{theorem}

\subsection{The algorithm}
Each fractional coordinate supplies two feasible alternatives to $x$. The algorithm compares their actual multilinear values, rather than trying to sample from the possibly unrealizable pair matrix.

For each $v$ with $0<x_v<1$, define
\begin{equation}\label{eq:conditionals}
 a^{(v)}=\frac{Y_{\cdot v}}{x_v},\qquad
 b^{(v)}=\frac{x-Y_{\cdot v}}{1-x_v}.
\end{equation}
These points lie in $B(M)$ by \eqref{eq:R}. In addition, $a_v^{(v)}=1$ and $b_v^{(v)}=0$.

\medskip
\noindent\fbox{\begin{minipage}{0.94\linewidth}
\textbf{Algorithm 1: conditional selection and one rounding call}
\begin{enumerate}[leftmargin=1.7em]
\item Form the candidate set $\mathcal Z=\{x\}\cup\{a^{(v)},b^{(v)}:0<x_v<1\}$.
\item Evaluate $F(z)$ for all $z\in\mathcal Z$. Choose a maximizer $z^\star$.
\item Apply Lemma~\ref{lem:round} to $z^\star$. Return the resulting basis.
\end{enumerate}
\end{minipage}}
\medskip

Algorithm 1 uses at most $2n+1$ candidate evaluations and one rounding call. It never computes a Gram factor, an eigenvalue, or a numerical matrix rank.

\subsection{Two edge inequalities}
The first inequality recovers the $3/4$ baseline. The second controls a weighted variance quantity; that bound will allow the spectral average to improve the baseline.

For this section, abbreviate $F=F(x)$ and $U=U(x,Y)$. Set
\[
 \sigma_v=\sqrt{x_v(1-x_v)},\qquad
 T=\sum_{uv\in H}w_{uv}\sigma_u\sigma_v.
\] Thus $\sigma_v^2=C_{vv}$; $T$ weights products of the marginal standard deviations along coverage edges.

Both estimates are edgewise, so their proof does not require any additional assumptions about the matroid.

\begin{lemma}\label{lem:scalar}
Every feasible point satisfies
\begin{equation}\label{eq:scalar}
 F\ge\frac34U,\qquad T\le F-\frac12U.
\end{equation}
If $T=0$, then $F=U$.
\end{lemma}
\begin{proof}
For one edge, put $a=x_u$, $b=x_v$, $y=Y_{uv}$, and $s=a+b$. Equation~\eqref{eq:frechet} gives
\[
 a+b-y\le\min(1,s).
\]
Also $ab\le s^2/4$. If $0\le s\le1$, then
$s-s^2/4\ge3s/4$. If $1\le s\le2$, then
$s-s^2/4\ge3/4$. Hence
\[
 a+b-ab\ge\frac34(a+b-y).
\]
Multiply by the nonnegative edge weight and sum.

For the second inequality, the arithmetic--geometric mean inequality gives
\begin{align*}
 \frac{a+b}{2}-ab
 &=\frac{a(1-b)+b(1-a)}{2}\\
 &\ge\sqrt{a(1-a)b(1-b)}.
\end{align*}
The edge contribution to $F-U/2$ is the left side plus $y/2$. It is at least the right side, since $y\ge0$. Again sum over edges.

If $T=0$, each positive-weight edge has $\sigma_u\sigma_v=0$. Positive semidefiniteness gives $|C_{uv}|\le\sigma_u\sigma_v=0$. Since
\begin{equation}\label{eq:FU}
 F-U=\sum_{uv\in H}w_{uv}C_{uv},
\end{equation}
we obtain $F=U$.
\end{proof}

\subsection{Conditional improvement}
We next compute the expected change in $F$ under a random choice between the two conditional vectors for a fixed coordinate. This is an artificial two-point distribution for analysis, not an assumed joint distribution with moments $(x,Y)$.

\begin{lemma}\label{lem:conditioning}
For $v$ with $0<x_v<1$, let $R_v$ equal $a^{(v)}$ with probability $x_v$, and $b^{(v)}$ otherwise. Then
\begin{align}
 \E R_v&=x,\label{eq:condmean}\\
 \E F(R_v)&=F-\frac{1}{\sigma_v^2}
       \sum_{ij\in H}w_{ij}C_{iv}C_{jv}.\label{eq:condF}
\end{align}
\end{lemma}
\begin{proof}
From $Y=xx^\top+C$,
\[
 a^{(v)}-x=\frac{C_{\cdot v}}{x_v},\qquad
 b^{(v)}-x=-\frac{C_{\cdot v}}{1-x_v}.
\]
The weighted deviations cancel, which gives $\E R_v=x$.

For coordinates $i,j$, the expected product of the deviations is
\begin{align*}
 \E[(R_{v,i}-x_i)(R_{v,j}-x_j)]
 &=C_{iv}C_{jv}\left(\frac{1}{x_v}+\frac{1}{1-x_v}\right)\\
 &=\frac{C_{iv}C_{jv}}{\sigma_v^2}.
\end{align*}
The linear terms in $F$ have unchanged expectations. Each negative quadratic term contributes the negative of this covariance, which proves \eqref{eq:condF}.
\end{proof}

\subsection{The spectral average}
The conditioning identity still depends on the chosen coordinate. We average it with objective-dependent weights so that the improvement becomes a matrix trace expression whose size can be bounded by rank.

Assume $T>0$. Let $J$ be the active set
\begin{equation}\label{eq:active}
 J=\{v:\sigma_v>0\text{ and some }u\text{ satisfies }uv\in H,
           \ w_{uv}\sigma_u>0\}.
\end{equation}
Thus $J$ contains the endpoints of edges with $w_{uv}\sigma_u\sigma_v>0$.

To express covariances as inner products, factor the positive semidefinite matrix $C$ as a Gram matrix. For every $v$ with $\sigma_v>0$, take a unit vector $g_v\in\R^d$ such that
\begin{equation}\label{eq:gram}
 C_{uv}=\sigma_u\sigma_v\ip{g_u}{g_v}.
\end{equation}
Such vectors exist because the squared norm of a Gram vector for $v$ is $C_{vv}=\sigma_v^2$.

Set $t_{uv}=w_{uv}\sigma_u\sigma_v$. In sums involving $g_u,g_v$, omit edges with $t_{uv}=0$. Define
\begin{equation}\label{eq:AD}
\begin{aligned}
 A&=\frac12\sum_{uv}t_{uv}(g_ug_u^\top+g_vg_v^\top),\\
 D&=\frac12\sum_{uv}t_{uv}(g_u+g_v)(g_u+g_v)^\top.
\end{aligned}
\end{equation}
These matrices package the endpoint directions and their edgewise interactions for the trace calculation. Both are positive semidefinite. They satisfy
\begin{equation}\label{eq:traces}
 \tr(A)=T,\qquad \tr(D)=T+F-U.
\end{equation}
Indeed, each $g_v$ is a unit vector, and \eqref{eq:FU} gives the second identity.

The incident weights below count every edge at its two endpoints. Their normalization makes the conditional objectives an average that the best candidate must dominate.

\begin{lemma}[Exact spectral identity]\label{lem:spectralidentity}
Define probabilities
\[
 p_v=\frac{\sum_{u:uv\in H}t_{uv}}{2T}\qquad(v\in J).
\]
First choose $v$ with probability $p_v$. Then choose $R_v$ as in Lemma~\ref{lem:conditioning}. The expected objective $Q$ is
\begin{equation}\label{eq:Qidentity}
 Q=F+\frac{\tr(A^2)-\tr(AD)}{T}.
\end{equation}
\end{lemma}
\begin{proof}
Every edge is counted twice in the incident-weight sum, so $\sum_vp_v=1$. Also,
\[
 \sum_vp_vg_vg_v^\top=\frac{A}{T}.
\]

Using \eqref{eq:gram}, the loss term in \eqref{eq:condF} becomes
\[
 \sum_{ij}t_{ij}\ip{g_i}{g_v}\ip{g_j}{g_v}.
\]
Its average over $v$ is
\[
 \frac1T\sum_{ij}t_{ij}g_i^\top A g_j.
\]
To express this average as a trace, introduce the symmetric cross-term matrix
\[
 K=\frac12\sum_{ij}t_{ij}(g_ig_j^\top+g_jg_i^\top).
\]
Expansion of \eqref{eq:AD} gives $D=A+K$. Since $A$ is symmetric,
\[
 \sum_{ij}t_{ij}g_i^\top Ag_j=\tr(AK)=\tr(AD)-\tr(A^2).
\]
Subtracting this averaged loss from $F$ proves \eqref{eq:Qidentity}.
\end{proof}

The exact average now yields the stronger guarantee. The ratio of two traces measures concentration across the objective-relevant covariance directions; ordinary rank provides a lower bound on that concentration. Here $C[J,J]$ denotes the principal submatrix indexed by the active set $J$.

\begin{theorem}[Effective-rank bound]\label{thm:effective}
Suppose $T>0$, and define
\[
 \beta=\frac{\tr(A^2)}{T^2},\qquad
 d_{\rm act}=\rank(C[J,J]).
\]
Algorithm 1 returns a basis satisfying
\begin{equation}\label{eq:effective}
 f(S)\ge\left(\frac34+\frac\beta4\right)U
 \ge\left(\frac34+\frac{1}{4d_{\rm act}}\right)U
 \ge\left(\frac34+\frac{1}{4d}\right)U.
\end{equation}
The value $\beta$ does not depend on the chosen Gram factorization. The quantity $1/\beta$ is the effective rank of $A$ in this theorem.
\end{theorem}
\begin{proof}
Let $\ell=\lambda_{\max}(A)/T$. The eigenvalues of $A/T$ are nonnegative and sum to one. If they are $\lambda_1,\ldots,\lambda_d$, then
\[
 \beta=\sum_i\lambda_i^2,
 \qquad
 \ell^2\le\beta\le\ell\le1.
\]

To control the cross term, use $D\succeq0$ and $\ell TI-A\succeq0$: 
\[
 \tr(AD)\le\ell T\tr(D).
\]
To justify this trace inequality, factor $D=D^{1/2}D^{1/2}$. The matrix
$D^{1/2}(\ell TI-A)D^{1/2}$ is positive semidefinite and has nonnegative trace.

Equations~\eqref{eq:Qidentity} and \eqref{eq:traces} yield
\begin{align}
 Q&\ge(1-\ell)F+\ell U-(\ell-\beta)T\nonumber\\
  &\ge(1-2\ell+\beta)F+\frac{3\ell-\beta}{2}U.\label{eq:Qstep}
\end{align}
The second step uses $T\le F-U/2$ and $\ell-\beta\ge0$. Furthermore,
\[
 1-2\ell+\beta\ge(1-\ell)^2\ge0.
\]
We can therefore use $F\ge3U/4$ in \eqref{eq:Qstep}. This gives
\[
 Q\ge\left[\frac34(1-2\ell+\beta)+\frac{3\ell-\beta}{2}\right]U
   =\left(\frac34+\frac\beta4\right)U.
\]
Algorithm 1 chooses a point with objective at least this average. Its final rounding does not decrease the objective.

For the rank refinements, write
$A=\sum_{v\in J}a_vg_vg_v^\top$, where
$a_v=\frac12\sum_u t_{uv}>0$.
Because every $a_v$ is positive, the range of $A$ is the span of the $g_v$ for $v\in J$. The positive diagonal scaling in \eqref{eq:gram} shows that its rank is $d_{\rm act}$. Cauchy--Schwarz applied to its nonzero normalized eigenvalues gives
\[
 1\le d_{\rm act}\sum_i\lambda_i^2=d_{\rm act}\beta.
\]
Also $d_{\rm act}\le d$.

Finally,
\[
 \tr(A^2)=\sum_{u,v\in J}a_ua_v\ip{g_u}{g_v}^{\!2}.
\]
All inner products are fixed by $C$ through \eqref{eq:gram}. Thus $\beta$ is independent of the Gram representation.
\end{proof}

\begin{proof}[Completion of Theorem~\ref{thm:main}]
If $d=0$, then $C=0$. Equation~\eqref{eq:covdiag} makes every $x_v$ zero or one. A zero-one point in $B(M)$ is a basis indicator: in any convex decomposition, every basis of positive weight must agree in each coordinate with $x$. Since $Y=xx^\top$, the two objectives agree.

If $d\ge1$ and $T=0$, Lemma~\ref{lem:scalar} gives $F=U$. Round $x$ or use Algorithm 1. The target factor is at most one. If $T>0$, Theorem~\ref{thm:effective} applies.

For rational input, all candidates in \eqref{eq:conditionals} are rational with polynomial encoding length. Their objectives can be compared exactly. Candidate construction uses $O(n^2)$ arithmetic operations, and direct objective evaluation uses $O(n|H|)$. Lemmas~\ref{lem:decomp} and \ref{lem:round} then give polynomial running time. After decomposition into $t\le n+1$ bases, the final merging stage uses $O(nr^2)$ independence queries. The square roots and spectral quantities occur only in the proof.

For real input, finite convex decompositions and the same finite exchange argument prove existence. No computational claim about arbitrary real encodings is needed.
\end{proof}

The covariance constraint $C\one=0$ also gives a bound in terms of the number of vertices.

\begin{corollary}[Finite-size consequence]\label{cor:finite}
For $n\ge2$ and any feasible point with nonzero covariance,
\[
 f(S)\ge\left(\frac34+\frac{1}{4(n-1)}\right)U(x,Y).
\]
This improvement tends to zero with $n$. It does not imply a fixed universal gain over $3/4$.
\end{corollary}
\begin{proof}
Use $d\le n-1$ from Proposition~\ref{prop:lift}.
\end{proof}

\section{A matching family of rational linear matroids}\label{sec:gap}
To determine whether the rank factor can be improved for this lift, we construct a feasible point whose promised rounded value already equals the integral optimum. The construction separates the matroid dependencies from the coverage edges.

\begin{theorem}[Sharpness at every positive covariance rank]\label{thm:gap}
For every integer $m\ge1$, there is a rationally represented rank-$2m$ matroid on $4m$ elements and a unit-weight graph of $m$ disjoint four-cycles such that
\[
 \OPT=3m+1,\qquad U^\star=4m.
\]
There is an optimal feasible point with covariance rank $m$ and $\beta=1/m$. Thus both the covariance-rank factor and the effective-rank factor are attained.
\end{theorem}

\subsection{The represented matroid}
The matroid will permit selecting both elements of one pair, but not both elements of two pairs. An explicit rational representation proves that these dependencies belong to a genuine matroid.

Put $k=2m$. Work in $\Q^k$, with standard basis $e_1,\ldots,e_k$. Let
$h=\sum_{j=1}^{k}e_j$. The elements occur in pairs
\[
 P_j=\{a_j,b_j\},\qquad j=1,\ldots,k,
\]
and their representing columns are
\begin{equation}\label{eq:representation}
 a_j=e_j,\qquad b_j=e_j+h.
\end{equation}
The rank is $k$, because all $a_j$ form a basis. A pair is \emph{complete} in a set when both its elements are selected.

We need all independent sets, not merely one displayed dependence, to compute the optimum and certify every conditional vector. The rank formula provides that complete description.

\begin{lemma}[Full independence characterization]\label{lem:family-ind}
A subset $S$ of these columns is independent if and only if $|S|\le k$ and $S$ contains at most one complete pair. Its rank is
\begin{equation}\label{eq:rankformula}
 \rho(S)=\min\{k,\ t(S)+\one[q(S)>0]\},
\end{equation}
where $t(S)$ counts the pairs that meet $S$ and $q(S)$ counts its complete pairs.
\end{lemma}
\begin{proof}
Any set of more than $k$ columns is dependent. If it contains two complete pairs $P_j,P_\ell$, then
\[
 b_j-a_j-b_\ell+a_\ell=0
\]
is a nontrivial dependence.

A set with one selected element from every pair is a basis. Its column matrix is $I+ht^\top$ for a vector $t\in\{0,1\}^k$, where $t_j=1$ precisely when $b_j$ is selected. To prove nonsingularity directly, suppose $(I+ht^\top)z=0$. Then $z=-h(t^\top z)$, so
$(1+t^\top h)(t^\top z)=0$. The scalar $1+t^\top h$ is positive. Therefore $z=0$.

Next take one complete pair $P_j$, omit a different pair $P_\ell$, and choose one element from every other pair. Subtract the columns of $P_j$ to obtain $h$. Subtract $h$ from any other selected $b_i$. The resulting span contains $h$ and every $e_i$ except $e_\ell$. These $k$ vectors are independent, since the $\ell$th coordinate of $h$ equals one. Thus the selected set is a basis.

A set with no complete pair extends to a basis with one element from each pair. A set with one complete pair and size at most $k$ has an empty pair. Leave one such pair empty and fill the others with one element. This extends the set to a basis with one complete pair and one empty pair. The characterization follows.

An independent subset of $S$ takes at most one element from each occupied pair, with at most one extra element in total. This gives the upper bound in \eqref{eq:rankformula}. Selecting one element from each occupied pair and, when available, one extra element attains that bound after truncation to size $k$.
\end{proof}

The dependence pattern also comes from a graphic construction. Its representing graph must be distinguished from the coverage graph.

\begin{remark}[Graphic interpretation]\label{rem:graphic}
Consider a separate graph $J_k=K_{2,k}$. Let its two hubs be $p,q$, and its other vertices be $u_1,\ldots,u_k$. Associate $P_j$ with edges $pu_j,qu_j$. An edge set is a forest exactly when it contains at most one complete pair. A cycle requires two complete pairs, and two such pairs form a four-cycle. Restricting these forests to at most $k$ edges gives precisely Lemma~\ref{lem:family-ind}. Thus our matroid is the rank-$k$ truncation of the graphic matroid of $J_k$. This does not assert that the truncated matroid is graphic. The graph $J_k$ is not the coverage graph defined next.
\end{remark}

\subsection{The coverage graph and its integral optimum}
The objective now rewards selecting a complete pair in each four-cycle. Since the matroid permits at most one complete pair overall, the basis classification determines the exact loss.

For each $i=1,\ldots,m$, put all four unit-weight edges between $P_{2i-1}$ and $P_{2i}$. There are no other edges. Write the vertex set of component $i$ as
\[
 V_i=P_{2i-1}\cup P_{2i}.
\]
Each component is $K_{2,2}$.

The following count establishes the integral benchmark against which the lifted witness will be compared.

\begin{lemma}\label{lem:opt}
The maximum feasible coverage value is $3m+1$.
\end{lemma}
\begin{proof}
In a component, selecting $a$ vertices on one side and $b$ on the other covers
\[
 4-(2-a)(2-b)
\]
edges. It suffices to examine bases. A basis has $k$ elements. If it has no complete pair, it selects one element from each pair and covers three edges in each component.

If a basis has one complete pair, it has exactly one empty pair. If these two pairs belong to the same component, that component contributes four edges. Every other component contributes three. Such a basis exists by Lemma~\ref{lem:family-ind}, and its value is $3m+1$.

If the two pairs belong to different components, the complete-pair component contributes four edges and the empty-pair component contributes two. The other components contribute three. The total is $3m$. There are no other types of bases. Hence the optimum is $3m+1$.
\end{proof}

\subsection{A feasible witness for the full conditional lift}
Assign one covariance direction to each four-cycle. We verify the full conditional constraints, not only the matrix’s positive semidefiniteness and rank.

Set $x_v=1/2$ for every element. For each $i$, let $q_i\in\R^{4m}$ have value $+1$ on $P_{2i-1}$, value $-1$ on $P_{2i}$, and value zero elsewhere. Define
\begin{equation}\label{eq:witness}
 C=\frac14\sum_{i=1}^{m}q_iq_i^\top,
 \qquad Y=xx^\top+C.
\end{equation}

We will certify the singleton vector and each conditional vector separately through explicit distributions over bases.

\begin{lemma}\label{lem:witness}
The pair in \eqref{eq:witness} is in $\cR(M)$. It has $\rank(C)=m$ and $U(x,Y)=4m$.
\end{lemma}
\begin{proof}
The vectors $q_i$ have disjoint supports and satisfy $q_i^\top q_i=4$. They are linearly independent. Thus $C\succeq0$ has rank $m$. Each diagonal entry of $C$ is $1/4$, so $Y_{vv}=1/2=x_v$. The block identity in Proposition~\ref{prop:lift} verifies $Z\succeq0$.

The point $x$ lies in $B(M)$. Choose one element from every pair, independently and with equal probabilities. Each outcome is a basis, and its mean is $x$.

Fix a vertex $v$ in component $i$. Its conditional point $Y_{\cdot v}/x_v$ with $v$ selected has coordinates $1$ on the pair containing $v$, $0$ on the other pair of $V_i$, and $1/2$ outside $V_i$. To realize these marginals, select the indicated complete pair, omit its partner in $V_i$, and choose one element from every remaining pair with equal probabilities. Every outcome is a basis by Lemma~\ref{lem:family-ind}.

The conditional point $(x-Y_{\cdot v})/(1-x_v)$ with $v$ excluded is realized by the same construction with the two sides of $V_i$ exchanged. This verifies all scaled-column constraints with actual convex combinations of bases.

Each coverage edge joins opposite sides of one component. Equation~\eqref{eq:witness} gives $C_{uv}=-1/4$ and $Y_{uv}=0$ on every such edge. Each edge contributes one to $U$, and there are $4m$ edges.
\end{proof}

\begin{proof}[Proof of Theorem~\ref{thm:gap}]
Lemma~\ref{lem:opt} gives the integral optimum. By \eqref{eq:frechet}, every edge contributes at most one to the lifted objective. Lemma~\ref{lem:witness} attains this bound, so $U^\star=4m$.

For the effective-rank calculation, use the standard basis of $\R^m$ and assign $g_v=e_i$ on $P_{2i-1}$ and $g_v=-e_i$ on $P_{2i}$. Here $\sigma_v=1/2$ and $t_{uv}=1/4$ on every edge. Equation~\eqref{eq:AD} gives
\[
 A=I_m,\qquad D=0,\qquad T=m,\qquad\beta=\frac1m.
\]

The witness has $F(x)=3m$. Both conditional points for each coordinate have objective $3m+1$: its fixed component contributes four, and the other components contribute three each. Algorithm 1 therefore reaches the integral optimum. Moreover,
\[
 \frac{\OPT}{U^\star}=\frac{3m+1}{4m}
       =\frac34+\frac{1}{4m}.
\]
This proves the claimed equalities and sharpness.
\end{proof}

Taking larger members of the family determines the worst-case ratio of the entire lift. The limiting statement is an infimum or supremum, not an assertion that a finite instance attains $3/4$.

\begin{corollary}[Exact unrestricted integrality gap]\label{cor:gap}
Over all graph--matroid instances with $\OPT>0$,
\[
 \inf\frac{\OPT}{U^\star}=\frac34,
 \qquad
 \sup\frac{U^\star}{\OPT}=\frac43.
\]
For every fixed $\delta>0$, some instance satisfies
$(3/4+\delta)U^\star>\OPT$.
\end{corollary}
\begin{proof}
For any real optimal lifted point, Lemma~\ref{lem:scalar} and the finite exchange argument give a basis of value at least $3U^\star/4$. This argument does not require a rational optimizer. Theorem~\ref{thm:gap} approaches equality as $m$ tends to infinity. For the last statement, take $m>1/(4\delta)$.
\end{proof}

\section{Pair-exclusion cuts and an exact repair}\label{sec:cuts}
The lift permits pair statistics that cannot hold simultaneously in a feasible set. We first express the missing compatibility condition as a valid inequality, then prove that this single inequality closes the gap on the constructed family.

\begin{definition}[Pair-exclusion family]
A collection of distinct two-element sets $P_1,\ldots,P_t$ is a pair-exclusion family when $P_i\cup P_j$ is dependent for all $i\ne j$. Write $Y_P=Y_{ab}$ when $P=\{a,b\}$.
\end{definition}

Pairwise exclusion means that at most one pair-selection event can occur in an independent set. Taking expectations gives the linear constraint on true pair moments.

\begin{proposition}[Validity of the clique inequality]\label{prop:cut}
Every actual distribution over independent sets satisfies
\begin{equation}\label{eq:cliquecut}
 \sum_{i=1}^{t}Y_{P_i}\le1
\end{equation}
for a pair-exclusion family. Given such a family, its pairwise exclusion property can be checked with at most $\binom t2$ independence queries.
\end{proposition}
\begin{proof}
An independent set cannot contain both $P_i$ and $P_j$, because its subset $P_i\cup P_j$ would then be independent. Thus at most one of the indicators $\one[P_i\subseteq S]$ is one. Take expectations. The query bound follows by testing each union.
\end{proof}

Proposition~\ref{prop:cut} proves this inequality directly. It is also valid for the exact monomial hull through the projection in Section~\ref{sec:monomials}; the citation to Fischer, Fischer, and McCormick~\cite{FFM22} concerns their monomial formulation, not an explicit statement of this inequality. Here the clique inequality separates the constructed witness from true moments. We do not give a polynomial-time method for finding a violated inequality among all such families.

For the constructed matroid, the union of any two complete pairs is dependent. The resulting inequality not only removes the displayed witness: the following certificate bounds every feasible point of the strengthened system.

\begin{theorem}[One cut makes the example family exact]\label{thm:repair}
For the instance in Theorem~\ref{thm:gap}, add the inequality
\begin{equation}\label{eq:familycut}
 \sum_{j=1}^{2m}Y_{a_jb_j}\le1
\end{equation}
to \eqref{eq:R}. The strengthened lifted optimum equals $3m+1$. For $m>1$, the witness in \eqref{eq:witness} violates this inequality by $m-1$.
\end{theorem}
\begin{proof}
The union of any two complete pairs is dependent by Lemma~\ref{lem:family-ind}. Hence Proposition~\ref{prop:cut} applies. In the witness, $Y_{a_jb_j}=1/2$ for every $j$. The left side of \eqref{eq:familycut} is $m$.

We give an explicit positive semidefinite certificate for the upper bound. Let $Z$ be the augmented moment matrix. For each coverage component $V_i$, define the fixed vector
\[
 h_i=\begin{pmatrix}-2\\\one_{V_i}\end{pmatrix}\in\R^{1+4m}.
\]
Every point satisfying the affine moment identities and $x(V)=2m$ obeys
\begin{equation}\label{eq:dualcertificate}
 \boxed{\quad
 3m+1-U(x,Y)
 =1-\sum_{j=1}^{2m}Y_{a_jb_j}
   +\frac12\sum_{i=1}^{m}h_i^\top Zh_i.
 \quad}
\end{equation}

To verify the identity, set
\[
 s_i=x(V_i),\quad
 c_i=\sum_{\substack{u\in P_{2i-1}\\v\in P_{2i}}}Y_{uv},\quad
 d_i=Y_{a_{2i-1}b_{2i-1}}+Y_{a_{2i}b_{2i}}.
\]
Then $U=\sum_i(2s_i-c_i)$ and $\sum_i s_i=2m$. Since $\diag(Y)=x$,
\[
 h_i^\top Zh_i=4-3s_i+2(c_i+d_i).
\]
Substituting this expression into the right side of \eqref{eq:dualcertificate} gives
$1-m+\sum_i c_i$, which is also $3m+1-U$.

The cut makes the first term on the right of \eqref{eq:dualcertificate} nonnegative. Positive semidefiniteness makes every remaining term nonnegative. Hence $U\le3m+1$. An optimal integral basis gives feasible moments for the strengthened system and attains this value. This proves exactness.
\end{proof}

\begin{remark}[Equivalent covariance certificate]
Since $h_i^\top Zh_i=(s_i-2)^2+\one_{V_i}^\top C\one_{V_i}$, identity~\eqref{eq:dualcertificate} is also
\[
 3m+1-U
 =1-\sum_jY_{a_jb_j}
 +\frac12\sum_i(s_i-2)^2
 +\frac12\sum_i\one_{V_i}^\top C\one_{V_i}.
\]
The version using $Z$ is linear in the lifted variables and supplies an explicit semidefinite dual upper-bound certificate. It does not rely on numerical optimization.
\end{remark}

\section{Exact input, numerical error, and reproducibility}\label{sec:implementation}
We distinguish exact feasible input, numerical approximation, and finite verification. Neither a numerical rank estimate nor a successful regression test substitutes for the theorem’s hypotheses.

\subsection{Exact input and certification}
Theorem~\ref{thm:main} rounds a supplied exact rational feasible point. It assumes neither that floating-point solver output is exactly feasible nor that an optimum has a polynomial-length rational representation.

A rational pair can be checked using the stated constraints. Positive semidefiniteness can be tested with exact rational elimination and Schur complements. Membership in the base polytope and in its scaled copies has polynomial-time oracle algorithms~\cite{Cunningham84,GLS93}. A zero scaling must be checked as equality to the zero vector, not by division.

For a direct certificate, supply a rational factorization $C=LDL^\top$ with $D$ here a nonnegative diagonal factor, not the matrix in \eqref{eq:AD}. Supply convex decompositions into bases for $x$ and the nontrivial conditional points. Exact matrix multiplication and independence queries then verify feasibility. The example family has such certificates in explicit form.

The covariance rank is discontinuous under perturbation. A numerical singular-value threshold cannot certify the exact integer $d$ in the theorem. In particular, small omitted eigenvalues do not justify substituting a smaller rank.

\subsection{Feasible numerical approximations}
Appendix~\ref{app:approximation} bounds the loss when rational candidate approximations are certified to lie in $B(M)$. It does not repair arbitrary infeasible solver output.

\subsection{Finite verification and its scope}
Appendix~\ref{app:finite-tests} records the tested witnesses, conditional constraints, matroid representations, rounding identities, populations, and boundary cases. These exact regression checks are not independent proof verification. The code rounds supplied decompositions; it is not a general decomposition oracle or semidefinite optimizer.

\section{Conclusion}
The full conditional positive semidefinite lift has a sharp limiting worst-case coverage ratio of $3/4$. A single-coordinate conditioning step gives a sharp refinement by covariance rank. The algorithm needs only exact evaluation of conditional candidates and one deterministic matroid rounding call.

For $m>1$, the gap witness does not arise from a distribution over bases, although each of its one-coordinate conditional marginal vectors is a convex combination of bases. The missing information concerns simultaneous selection of complete pairs whose union is dependent. One valid pair-exclusion inequality removes the witness and makes the example family exact.

The paper does not prove a better universal approximation ratio for arbitrary matroid-constrained edge coverage. It also does not prove efficient separation of all pair-exclusion inequalities, an approximation ratio for the strengthened lift on arbitrary instances, or an efficient method to obtain low-rank optimal moments.

\noindent\textbf{Declaration of Generative AI Use.} ChatGPT (OpenAI) was used to assist with mathematical analysis and manuscript preparation. The author takes responsibility for the final content.

\label{rev:main-end}
\appendix
\section{Matroid facts used by the rounding algorithm}\label{app:matroids}
This appendix supports the rounding guarantee in Section~\ref{sec:rounding}. We use the rank function $\rho$ of the original matroid from Section~\ref{sec:definitions} and prove the exchange and polytope facts needed by Appendix~\ref{app:constructive-rounding}. The constructive rational-polytope algorithm itself remains the cited prerequisite in Lemma~\ref{lem:decomp}.

Rank submodularity controls closure, which is the tool used below to find a symmetric exchange.

\begin{lemma}[Rank submodularity]\label{lem:ranksub}
For subsets $A,B\subseteq V$,
\[
 \rho(A)+\rho(B)\ge\rho(A\cup B)+\rho(A\cap B).
\]
\end{lemma}
\begin{proof}
Choose a basis $I$ of $A\cap B$. Extend it to a basis $J$ of $A$, then to a basis $K$ of $A\cup B$. All elements added in the second extension can be taken from $B\setminus A$. No element of $A\setminus J$ can augment $J$, and dependence remains after adding more elements. The set $I\cup(K\setminus J)$ is independent and lies in $B$. Therefore,
\[
 \rho(A\cap B)+\rho(A\cup B)-\rho(A)
 =|I|+|K\setminus J|\le\rho(B).
\]
\end{proof}

Define the closure of $A$ by
$\cl(A)=\{v:\rho(A+v)=\rho(A)\}$.
Rank submodularity implies the following diminishing-increment fact: if $A\subseteq B$ and $v\in\cl(A)$, then $v\in\cl(B)$. It also implies that $\rho(\cl(A))=\rho(A)$, by adding closure elements one at a time. Consequently, $\cl(\cl(A))=\cl(A)$.

\begin{proof}[Rank-function proof of Lemma~\ref{lem:exchange}]
Take bases $B_1,B_2$ and $i\in B_1\setminus B_2$. We seek $j\in B_2\setminus B_1$ for which both $B_1-i+j$ and $B_2-j+i$ are bases. Let $H=\cl(B_1-i)$. Then $\rho(H)=r-1$ and $i\notin H$. Choose an inclusion-minimal set $K\subseteq B_2$ that spans $i$, meaning $i\in\cl(K)$. It exists because $B_2$ is a basis. It is nonempty, because $i$ belongs to the basis $B_1$ and is not a loop.

The set $K$ is not contained in $H$. Otherwise the closure properties would imply $i\in\cl(H)=H$. Choose $j\in K\setminus H$. Since $B_1-i\subseteq H$, this element lies in $B_2\setminus B_1$. It augments $B_1-i$, so $B_1-i+j$ is a basis.

We show that $B_2-j+i$ is also a basis. Suppose not. Its rank would be $r-1$, because it contains the independent set $B_2-j$. Since $K$ minimally spans $i$, the set $(K-j)+i$ is independent and has size $|K|$. Apply rank submodularity to $K+i$ and $(B_2-j)+i$. Their union is $B_2+i$, and their intersection is $(K-j)+i$. This gives
\[
 |K|+(r-1)\ge r+|K|,
\]
which is impossible. Thus both exchanges are feasible.
\end{proof}

The decomposition algorithm needs a rational description with bounded coefficients. The rank inequalities supply that description; their validity as the full base polytope follows by comparing all linear objectives with greedy.

\begin{lemma}[Base-polytope rank description]\label{lem:basepoly}
The base polytope satisfies
\[
 B(M)=\{z\in\R^V:z\ge0,\ z(V)=r,\ z(A)\le\rho(A)\text{ for all }A\subseteq V\}.
\]
\end{lemma}
\begin{proof}
Every basis indicator satisfies these constraints. Let $P$ denote the right side. It is bounded, since the singleton inequalities give $0\le z_v\le1$.

Take any linear objective $c$. Add the same constant to every coefficient so that the coefficients are nonnegative. This does not change maximizers on $P$ or $B(M)$, because all points have coordinate sum $r$. Order the vertices so that $c_1\ge\cdots\ge c_n\ge0$, and set $c_{n+1}=0$. Let $A_i=\{1,\ldots,i\}$. For $z\in P$,
\[
 c^\top z=\sum_{i=1}^{n}(c_i-c_{i+1})z(A_i)
 \le\sum_{i=1}^{n}(c_i-c_{i+1})\rho(A_i).
\]

Greedy produces a basis $B$ with $|B\cap A_i|=\rho(A_i)$ for every prefix. To see this, after scanning a prefix its accepted set is maximal independent in that prefix. All maximal independent subsets of a matroid restriction have its rank. Hence the last upper bound equals $c^\top\one_B$.

Every linear objective therefore has the same maximum on $P$ and $B(M)$. The separating-hyperplane theorem for compact convex sets gives $P=B(M)$.
\end{proof}

\section{Constructive rounding and rational complexity}\label{app:constructive-rounding}
This appendix proves Lemma~\ref{lem:round} from Section~\ref{sec:rounding}. Given $z\in B(M)$, we obtain an explicit convex combination of bases and merge its terms without decreasing $F$. The rank and exchange facts used here were established in Appendix~\ref{app:matroids}. Input length includes every rational numerator and denominator.

A finite convex decomposition is the input to the merging procedure. The following standard oracle result supplies one with polynomial encoding length.

\begin{lemma}[Convex decomposition of a base-polytope point]\label{lem:decomp}
Given an independence oracle for $M$ and $z\in B(M)\cap\Q^V$, one can obtain
\[
 z=\sum_{j=1}^{t}\lambda_j\one_{B_j},\qquad
 \lambda_j>0,\quad \sum_j\lambda_j=1,\quad B_j\in\cB,
 \quad t\le n+1,
\]
with rational coefficients in polynomial time. The coefficient encodings have polynomial length.
\end{lemma}
\begin{proof}
Linear optimization over bases is implemented by greedy, as shown in the proof of Lemma~\ref{lem:basepoly}. The argument permits negative coefficients: all bases have size $r$, so a common shift makes the coefficients nonnegative without changing their maximizers.

Apply the standard optimization--separation equivalence and constructive Carath\'eodory theorem for well-described rational polytopes~\cite[Chapter 6]{GLS93}. Each rank inequality has zero-one coefficients and an integer right side at most $n$, so its encoding length is polynomial. The number of inequalities can be exponential; the argument uses a well-described rational-polytope oracle representation, not a polynomial-size explicit inequality list. Its vertices are zero-one vectors; an explicit rank-inequality description is proved in Appendix~\ref{app:matroids}. The constructive theorem returns at most $n+1$ vertices and rational convex coefficients of polynomial encoding length. This use of general rational-polytope algorithms is the algorithmic prerequisite; it is not an enumeration of all bases.
\end{proof}

To merge two bases, we need an exchange that is feasible in both directions. Symmetric exchange provides the two endpoints whose objective values the algorithm compares.

\begin{lemma}[Symmetric basis exchange]\label{lem:exchange}
Let $B_1,B_2$ be bases and $i\in B_1\setminus B_2$. Some $j\in B_2\setminus B_1$ satisfies
\[
 B_1-i+j\in\cB,\qquad B_2-j+i\in\cB.
\]
Such an element can be found using at most $2r$ independence queries.
\end{lemma}
\begin{proof}
This is the symmetric basis-exchange property. A rank-function proof appears in Appendix~\ref{app:matroids}. Test the at most $r$ elements of $B_2\setminus B_1$. Each test uses two independence queries on sets of size $r$.
\end{proof}

The two prerequisites now give the full proof of the main-text rounding guarantee.
\begin{proof}[Proof of Lemma~\ref{lem:round}]
Take the decomposition in Lemma~\ref{lem:decomp}. Keep its weights fixed while changing bases. Choose two terms $\lambda\one_{B_1}$ and $\mu\one_{B_2}$ with $\lambda,\mu>0$. If their bases differ, apply Lemma~\ref{lem:exchange} to obtain $i,j$. Replacing only $B_1$ gives the point
\[
 z^-=z-\lambda e_i+\lambda e_j.
\]
Replacing only $B_2$ gives
\[
 z^+=z+\mu e_i-\mu e_j.
\]
Both points remain convex combinations of bases. Also,
\[
 z=\frac{\mu}{\lambda+\mu}z^-
    +\frac{\lambda}{\lambda+\mu}z^+.
\]

Along a direction $e_i-e_j$, the polynomial $F$ is convex. Its quadratic coefficient is $w_{ij}\ge0$, with $w_{ij}=0$ if $ij\notin H$. Thus
\[
 F(z)\le\max\{F(z^-),F(z^+)\}.
\]
Evaluate both values exactly and keep an endpoint with the larger value.

Each exchange reduces $|B_1\setminus B_2|$ by one. After at most $r$ exchanges, the bases coincide. Replace the two terms by their common basis with weight $\lambda+\mu$. Continue until one term remains. At every step, the fractional objective does not decrease. The final vector is a basis indicator, so its objective is $f(B)$.

There are at most $(t-1)r$ exchanges. Each uses at most $2r$ independence queries and two evaluations of the explicit quadratic objective. Rational weights are only added. All intermediate coordinates are sums of the original rational weights. Their encoding length stays polynomial. This proves the computational claim.
\end{proof}

This is an objective-aware deterministic version of standard exchange rounding~\cite{CCPV11,CVZ09}. The proof uses the explicit form of $F$. No value-oracle sampling is required.

\section{Feasible candidate approximations}\label{app:approximation}
This appendix supports Section~\ref{sec:implementation}. The original lifted point remains exact; only its conditional candidate vectors are approximated, and the supplied approximations must remain in $B(M)$. We first control the objective loss, then propagate it through candidate selection and rounding.

The derivative of each edge contribution bounds its sensitivity to a coordinate perturbation. Summing those sensitivities yields a uniform error bound in terms of total edge weight $W$.

\begin{lemma}[Lipschitz estimate]\label{lem:lipschitz}
For $z,z'\in[0,1]^V$,
\[
 |F(z)-F(z')|
 \le\sum_v\deg_w(v)|z_v-z'_v|
 \le2W\|z-z'\|_\infty,
\]
where $\deg_w(v)=\sum_{u:uv\in H}w_{uv}$.
\end{lemma}
\begin{proof}
The partial derivative is
$\partial F/\partial z_v=\sum_{u:uv\in H}w_{uv}(1-z_u)$.
It lies between zero and $\deg_w(v)$ throughout the cube. Integrate along the line segment from $z$ to $z'$. The final inequality uses $\sum_v\deg_w(v)=2W$.
\end{proof}

Apply the Lipschitz bound to the candidate that certifies the exact theorem. Selecting the best supplied approximation cannot be worse than selecting the approximation to that candidate. Thus one perturbation bound suffices.

\begin{proposition}[Certified approximation of candidate points]\label{prop:robust}
Let $(x,Y)$ be an exact real feasible point. Let $\mathcal Z$ be its candidates from Algorithm 1. Suppose that, for every $z\in\mathcal Z$, a rational point $\widehat z\in B(M)$ is supplied with
$\|z-\widehat z\|_\infty\le\eta$.
Choose the largest exact value $F(\widehat z)$ and round that point. If $d\ge1$, the output satisfies
\[
 f(S)\ge\left(\frac34+\frac{1}{4d}\right)U(x,Y)-2W\eta.
\]
The factor $3/4+\beta/4$ can replace the rank factor when $T>0$. For $d=0$, the corresponding factor is one.
\end{proposition}
\begin{proof}
The exact candidate set contains a point of value at least the applicable factor times $U$. Its supplied feasible approximation loses at most $2W\eta$ by Lemma~\ref{lem:lipschitz}. Selecting the best approximate candidate and then rounding cannot decrease that lower bound.
\end{proof}

If the underlying feasible point also satisfies $U(x,Y)\ge U^\star-\varepsilon$, Proposition~\ref{prop:robust} gives
\[
 f(S)\ge\alpha\OPT-\alpha\varepsilon-2W\eta,
\]
where $\alpha$ is the applicable pointwise factor. The hypotheses require certified feasible points and a certified rank bound. This statement is not a feasibility-repair algorithm for arbitrary solver output.

\section{Finite verification and its scope}\label{app:finite-tests}
The tests check the witness, the matroid representation, and the rounding identities on stated finite inputs. Their counts describe regression coverage, not independent instances of a general proof.

The accompanying programs use integer and rational arithmetic. They need the Python standard library only. No numerical eigensolver or semidefinite solver is used.

\begin{table}[ht]
\centering
\small
\begin{tabular}{@{}rrrrrrr@{}}
\toprule
$m$ & $n$ & $r$ & $d$ & $\OPT$ & $U^\star$ & Rank-sized subsets \\
\midrule
1 & 4 & 2 & 1 & 4 & 4 & 6 \\
2 & 8 & 4 & 2 & 7 & 8 & 70 \\
3 & 12 & 6 & 3 & 10 & 12 & 924 \\
4 & 16 & 8 & 4 & 13 & 16 & 12,870 \\
5 & 20 & 10 & 5 & 16 & 20 & 184,756 \\
\bottomrule
\end{tabular}
\caption{Exact checks for the gap family. The value $U^\star$ is certified by the explicit feasible witness and the edgewise upper bound, not by an SDP solve.}
\label{tab:tests}
\end{table}

The family verifier examines 198,626 rank-sized subsets in total. It checks 120 conditional convex decompositions and the exact augmented Gram identity. For $m\le3$, it also checks all rank-sized subsets using determinants of the explicit representation.

A separate structure verifier computes column ranks for all 4,368 subsets across $m=1,2,3$. It checks both the independence characterization and \eqref{eq:rankformula}.

The rounding verifier passes 97 deterministic test cases. These include 18 exact effective-rank checks and 33 checks of \eqref{eq:dualcertificate}. 

The cases include the nonrealizable gap witnesses, rational mixtures with actual basis moments, uniform matroids, partition matroids, and represented matroids. Boundary cases include rank zero, integral points, loops, coloops, zero edge weights, and $T=0$ with nonzero covariance. The pseudorandom seed for the generated cases is 20261009.

These are regression checks. They are not a formal proof-assistant verification or independent peer review. The general proofs above do not depend on them. The code implements deterministic merging for a supplied decomposition. It does not implement a general base-polytope decomposition oracle or a general SDP optimizer.

\bibliographystyle{unsrt}
\bibliography{references}
\end{document}